\documentclass[12pt]{amsart}
\usepackage{amssymb}
\usepackage{amsfonts}
\usepackage{amsmath}
\usepackage{amsthm}
\usepackage{graphicx}
\usepackage{epstopdf}
\usepackage{polynom}

\newtheorem{theorem}{Theorem}[section]
\newtheorem{lemma}[theorem]{Lemma}
\newtheorem{proposition}[theorem]{Proposition}
\newtheorem{corollary}[theorem]{Corollary}

\theoremstyle{definition}

\newenvironment{definition}[1][Definition]{\begin{trivlist}
\item[\hskip \labelsep {\bfseries #1}]}{\end{trivlist}}
\newenvironment{remark}[1][Remark]{\begin{trivlist}
\item[\hskip \labelsep {\bfseries #1}]}{\end{trivlist}}

\newcommand{\ds}{\displaystyle}

\DeclareMathOperator{\rank}{Rank}

\title{Introducing the Polylogarithmic Hardy Space}
\author{Joel A. Rosenfeld}

\begin{document}
\maketitle

\section{Introduction}
\label{intro}
This paper studies the Polylogarithmic Hardy space given by $$PL^2 = \left\{ f(z,s) = \sum_{n=1}^\infty a_n \frac{z^n}{n^s} : \sum_{n=1}^\infty |a_n| < \infty\right\}.$$ This space has a strong connection to the polylogarithm $L_s(z) = \sum_{n=1}^\infty z^n n^{-s}$ in that it is a reproducing kernel Hilbert space with kernel functions $K(z,w,s,t) = L_{s+\bar t}(z\bar w)$.

In the past two decades, the study of reproducing kernel Hilbert spaces (RKHSs) has been extended to include the Hardy space of Dirichlet series. That is the space, $\mathcal{H}^2$, of functions that are representable as Dirichlet series with square summable coefficients. This space was introduced by Hedenmalm et al. in \cite{hedenmalm}. In \cite{hedenmalm}, the multiplication operators for this space were characterized as functions expressable as Dirichlet series bounded in the right half plane. Their results were extended by McCarthy in \cite{mccarthy} who studdied the so-called weighted Hardy spaces of Dirichlet series, and McCarthy expressed a weight for which the multiplication operators satisfied the Nevanlinna-Pick Interpolation property \cite{AgMc}. Other notable works in the space are due to Bayart, who studied the composition operators on $\mathcal{H}^2$ in \cite{Bayart}, as well as the study of a family of shift operators by Olofsson in \cite{Olofsson}.

The classical Hardy space $H^2$ has itself been studied since its introduction by Riesz in 1923 \cite{riesz}. This is the space of complex valued function analytic in the disc $\mathbb{D}$ with square summable coefficients. This is the standard example of a RKHS, and details about it can be found in many texts \cite{AgMc,garnett, hoffman}. In particular, the bounded multiplication operators over the Hardy space is characterized as the collection of analytic functions bounded in the disc \cite{AgMc}. Sarason in \cite{sarason} classified the densely defined multiplication operators over the Hardy space as those functions inside of the Smirnov class, $N^+$. This was the first characterization of densely defined multiplication operators over a RKHS.

The study of Toeplitz operators over the Hardy space is also an important area of research. This is the collection of operators where multiplication by an $L^\infty(\mathbb{T})$ function is followed by the projection onto the Hardy space. Bounded Toepliz operators are characterized by their matrix representation, where an operator is Toeplitz iff its matrix representation is constant down each diagonal and the entries are the Fourier series of an $L^\infty$ function. The class of Toeplitz operators arises by recognizing that $H^2$ can be identified with a subspace of $L^2(\mathbb{T})$ through the boundary values of functions in $H^2$.

The polylogarithm has been studied as far back as the time of Leibniz, who studied the dilogarithm, $L_2(z)$, in his letters to Johann Bernoulli hoping to resolve the Basel problem [Mathematische Schriften III]. It has since then found many applications in number theory, analysis, and physics. The study of the properties of the polylogarithm has been carried out over the past century and can be found in the works of Don Zagier as well as others \cite{zagier3,lewin,zagier2,zagier1}.

In statistical mechanics the polylogarithm is used to count the average energy in a boson gas \cite{statmech}. The density function for a boson gas is given by $n(x) = x^{p-1}(e^{(x-\mu)}-1)^{-1}$. Here $n$ is the energy density of bosons at the energy level $x$. The term $\mu$ is the chemical potential energy, and is always non-positive. The average energy of particles in a boson gas is then given by $$\Gamma(s) L_s(e^{\mu}) = \int_0^\infty \frac{x^{s-1}}{e^{(x-\mu)} - 1} dx$$ when $s = 2$.

Despite the widespread use of the polylogarithm, there does not seem to be any literature studying functions of the form $\sum_{n=1}^\infty a_n \frac{z^n}{n^s}$. This stands in contrast with the study of Dirichlet series as well as power series. This paper aims to develop the theory of Hilbert spaces of polylogarithm-like series. In particular, this paper will develop some of the analytic theory of such functions and then discuss function theoretic operators on this space.

In section \ref{isometry}, we will demonstrate that the Bose-Einstein integral formula can be modified to provide an isometric isomorphism between the classical Hardy space, $H^2$, and the Polylogarithmic Hardy space. The multiplication operators over the polylogarithmic Hardy space will be investigated in section \ref{sec_multipliers}. In particular it will be shown that there are no nontrivial bounded or densely defined multiplication operators. The space $PL^2$ is then the first RKHS known for which there are only trivial densely defined multiplication operators. Section \ref{sec_toep} views $PL^2$ as a subspace of $H^2 \times \mathcal{H}^2$, and investigates Toeplitz compressions over $PL^2$. Through this study, a proof of the multiplicativity of the divisor function is found through the decomposition of a class of operators.

\section{Properties of the Polylogarithmic Hardy Space}
\label{isometry}
This section serves as motivation into the study of the Polylogarithmic Hardy space defined in the introduction. Recall that the classical Hardy space is defined to be the collection of functions analytic in the disc with square summable Taylor coefficients. The Hardy space is a reproducing kernel Hilbert space (RKHS) with kernel function $\hat k_\lambda(z) = (1-\bar \lambda z)^{-1}$, which is known as the Szeg\H{o} kernel. For a boson gas, the chemical potential energy is nonpositive, and hence $0 <e^{\mu} < 1$. If we let $\lambda = e^{\mu}$, the Bose-Einstein integral can be rewritten as follows: $$\Gamma(s) L_s(\lambda) = \int_0^\infty \frac{x^{s-1}}{e^{(x-\mu)} - 1} dx$$ $$= \int_0^\infty x^{s-1} \frac{\lambda e^{-x} }{1-\lambda e^{-x}} dx = \int_0^\infty x^{s-1} \lambda e^{-x} \hat k_\lambda(e^{-x})dx.$$ The integral can be viewed as the Mellin transform of the Szeg\H{o} kernel preceeded by multiplication by $z$ and composition with $\lambda e^{-x}$. Thus the Bose-Einstein integral can be transformed into an isometry between Hilbert spaces as follows:

\begin{theorem}\label{isothm}Let $f(z,s) = \sum_{n=1}^\infty a_n z^n n^{-s} \in PL^2$. If $g(z) = \sum_{n=0}^\infty a_{n+1} z^n$ is in $H^2$ then $$f(z,s) = \frac{1}{\Gamma(s)} \int_0^\infty x^{s-1} ze^{-x} g(ze^{-x})dx := W(g)$$ where $Re(s) > 0$. Moreover, $W$ is an isometric isomorphism between $H^2$ and $PL^2$ with inverse map $f(z,s) \mapsto S^* f(z,0)$ (where $S = M_z$ is the shift operator).\end{theorem}

\begin{proof} The proof of the theorem can be accomplished by an application of the dominated convergence theorem and the definition of the $\Gamma$ function. Consider the integral,
$$\frac{1}{\Gamma(s)} \int_0^\infty x^{s-1}(ze^{-x})g(ze^{-x}) dx
 =  \ds \frac{1}{\Gamma(s)} \int_0^\infty x^{s-1} \sum_{n=1}^\infty a_{n} z^{n} e^{-nx} dx$$
$$ =  \ds \frac{1}{\Gamma(s)} \sum_{n=1}^\infty a_n \frac{z^n}{n^{s-1}} \int_0^\infty (nx)^{s-1} e^{-nx} dx
 =  \ds \frac{1}{\Gamma(s)} \sum_{n=1}^\infty a_n \frac{z^n}{n^s} \Gamma(s)$$
$$=\ds \sum_{n=1}^\infty a_n \frac{z^n}{n^s} = f(z,s).$$
\end{proof}

We see from Theorem \ref{isothm} that $W$ is an isometric isomorphism that maps $z^n$ to $z^{n+1}(n+1)^{-s}$ and it preserves the Taylor coefficients.

Each function in $PL^2$ is a function of two complex variables that is analytic for $|z| < 1$ and entire in $s$. Each function in $PL^2$ satisfies the relation $z \cdot \frac{\partial}{\partial z} f(z,s+1) = f(z,s)$, which is a property these functions share with the classical polylogarithm.

In addition, $PL^2$ is a RKHS with the inner product $$\langle f(z,s), g(z,s) \rangle_{PL^2} = \langle S^*f(z,0), S^* g(z,0) \rangle_{H^2},$$ for $f,g \in PL^2$, and with the kernel function $$K(z,w,s,t) = L_{s+\bar t}(z\bar w).$$ The space $PL^2$ can be determined to be a RKHS directly via Cauchy-Schwarz or by recognizing that $PL^2$ is a closed subspace of $H^2 \times \mathcal{H}^2$, the RKHS resulting from taking the tensor product of the classical Hardy space and the Hardy space of Dirichlet series.

Through its reproducing kernels, it can been seen that $PL^2$ contains subspaces with a natural identification with other RKHSs. For instance if $s,t=-1/2$ are fixed, the kernel functions become $K(z,w,-1/2,-1/2) = \sum_{n=1}^\infty n z^n \bar w^n = \tilde K(z,w,-1/2,-1/2) - 1$, where $\tilde K$ is the kernel function of the Dirichlet space, $\mathcal{D} = \{ f(z) = \sum_{n=0}^\infty a_n z^n : \sum_{n=0}^\infty n|a_n| < \infty \}$. For other values of $s$ the Bergman space and Hardy space can be found as well. When $z=1$, it can be seen that $PL^2$ also contains $\mathcal{H}^2$ as a subspace.

We now develop a lemma about the continuity of functions in $PL^2$ as well as a lemma about representations of the recipricols of functions in $PL^2$. These lemmas will be necessary for the development in Section \ref{sec_multipliers}.

\begin{lemma}\label{cn_prop}Let $z_0 \in \mathbb{D}$, $\gamma > 0$, and $f(z,s) \in PL^2$. For every $\epsilon > 0$ there is a $\delta > 0$ so that for all  $z \in B_\delta(z_0)$ and all $s$ for which $\sigma = Re(s) \ge 3/2 + \gamma$ we have $|f(z,s)-f(z_0,s)| < \epsilon$.\end{lemma}

\begin{proof}Fix $x_0 \in \mathbb{D}$ and $\epsilon > 0$. Now consider
$$\begin{array}{rcl}\ds |f(z,s)-f(z_0,s)|
& = & \ds \left| \sum_{k=1}^\infty a_k \frac{z^k - z_0^k}{k^s} \right|\vspace{.1in}\\
& = & \ds |z-z_0| \left| \sum_{k=1}^\infty a_k \frac{z^{k-1}z_0^0 + z^{k-2}z_0^1 + \cdots + z^0 z_0^{k-1}}{k^s} \right|\vspace{.1in}\\
& = & \ds |z-z_0| \left| \sum_{k=1}^\infty a_k \frac{1}{k^{\sigma -1 }} \right|\vspace{.1in}\\
& = & \ds |z-z_0| \left( \sum_{k=1}^\infty |a_k|^2 \right)^{1/2} \left( \sum_{k=1}^\infty \frac{1}{k^{2\sigma-2}}\right)^{1/2}\vspace{.1in}\\
& = & \ds |z-z_0| \cdot \|f\| \cdot \zeta(2\sigma - 2)^{1/2}\vspace{.1in}\end{array}.$$

For $\sigma = Re(s) \ge 3/2+\gamma$ the inequality yields $$|f(z,s) - f(z_0, s)| < |z-z_0| \cdot \|f\| \cdot \zeta(1+2\gamma)^{1/2}.$$ Since $\zeta(1+2\gamma) \neq 0$, $\delta$ may be chosen so that $0 < \delta < \epsilon \|f\|^{-1} \zeta(1+2\gamma)^{-1/2}$, and we have proved the theorem. 
\end{proof}

For the next lemma we must first recall the Weiner type theorem of Hewitt and Williamson \cite{hewitt}.

\begin{theorem}[Hewitt and Williamson]\label{hwthm} Suppose that $\sum_{n=1}^\infty |a_k| < \infty$. If the Dirichlet series $f(s) = \sum_{n=1}^\infty a_n n^{-s}$ is bounded away from zero in absolute value on $\sigma \ge 0$ then $(f(s))^{-1} = \sum_{n=1}^\infty b_n n^{-s}$ for $\sigma \ge 0$ where $\sum_{n=1}^\infty |b_n| < \infty$ and this series converges for $\sigma \ge 0$.\end{theorem}

The result found in \cite{hewitt} is more general than the above statement, however this will suit the purposes of this paper. Recall that if $f(s)= \sum_{n=1}^\infty a_n n^{-s}$ converges for some $s$ (and $a_1 \neq 0$), then there is some half plane where $f(s)$ is nonvanishing. Moreover $f(s) \to a_1$ as $\sigma = Re(s) \to \infty$. This means when $a_1 \neq 0$ we can graduate the result of Theorem \ref{hwthm} to this nonvanishing half-plane for $f$ and not just $\sigma \ge 0$. Other facts concerning Dirichlet series can be found in \cite{apostol, hardy}.

\begin{lemma}\label{inverse} Let $K$ be a compact subset of the punctured disc $\mathbb{D}\setminus\{0\}$. If $f(z,s)=\sum_{n=1}^\infty a_n z^n n^{-s} \in PL^2$ with $a_1 \neq 0$, then there is a real number $\sigma_K$ and $\alpha_n : K \to \mathbb{C}$ for which $(f(z,s))^{-1} = \sum_{n=1}^\infty \alpha_n(z) n^{-s}$ for all $z \in K$ and $\sigma = Re(s) > \sigma_K$.\end{lemma}

\begin{proof}Fix a real number $\delta$ for which $0 < \delta < |a_1| \min_{K}|z|$. For any fixed $z_0 \in K$, there is a half plane $\sigma \ge \sigma_{z_0}$ for which $f(z_0,s)|>\delta >0$.

By virtue of Lemma \ref{cn_prop}, $\sigma_z$ can be chosen so that it varies continuously with respect to $z$ (possibly requiring $\sigma_z \ge 3/2+\gamma$ with $\gamma > 0$ fixed). In particular, $\sigma_z$ takes a maximum value on $K$, $\sigma_K$. In every half plane $\sigma > \sigma_{K}$ the function $|f(z,s)|$ is bounded away from zero. For every $z \in \mathbb{D}$, $\sum_{n=1}^\infty |a_n z^n|<\infty$. Now apply Theorem \ref{hwthm} to obtain the representation for $(f(z,s))^{-1}$ and complete the proof.
\end{proof}

The functions $\alpha_n$ in Lemma \ref{inverse} can actually be shown to be polynomials in $z$. It can be demonstrated via multiplication by $f$ and by using the uniqueness of Dirichlet series.

\section{Multiplication Operators on $PL^2$}
\label{sec_multipliers}
In this section we consider the most basic of function theoretic operators over $PL^2$. In particular, this section investigates multiplication operators over $PL^2$. Given a function $\phi: \mathbb{D} \times \mathbb{C} \to \mathbb{C}$, the multiplication operator with symbol $\phi$ is given by $M_\phi f=\phi f$ whenever $\phi f \in PL^2$. The domain of $M_\phi$ is the set $D(M_\phi) = \{ f \in PL^2: \phi f \in PL^2\}$. Multiplication operators over RKHSs are closed operators \cite{frank}. This means $D(M_\phi)=PL^2$ iff $M_\phi$ is bounded, by the closed graph theorem.

Multiplication operators play an important role in linear systems theory, since they can be viewed as transfer functions of linear systems. They have been well studied in the context of RKHSs in particular. Densely defined multiplication operators, those multiplication operators for which $D(M_\phi)$ is dense, arise in quantum mechanics as well as in the study of RKHSs. The classification problem for densely defined multiplications operators has not been as extensively studied as the classification of bounded multiplications operators. However, the densely defined multiplications operators have been completely classified over the Hardy space \cite{sarason} as well as the Sobolev space $W^{1,2}[0,1]$ \cite{rosenfeld}.

For spaces of entire functions, there are no nonconstant bounded multiplication operators. This follows since if $\phi$ is the symbol of a bounded multiplication operator over a space of entire functions, then it is entire itself. Moreover, for every bounded multiplication operator, $\sup_{z \in \mathbb{C}} |\phi(z)| \le \| M_\phi\|$. Thus the symbol is a bounded entire function and is therefore constant. Thus the Fock space of entire functions has only constant bounded multiplication operators \cite{AgMc,zhu}. However, the Fock space does have a rich space of densely defined multiplication operators. On particularly important densely defined multiplication operator is the operator $M_z$ \cite{zhu}.

The following theorems will demonstrate that $PL^2$ has only constant bounded multipliers, and that it in fact has only constant densely defined multiplication operators. This is the first space found to have this property, and it is surprising that such a space exists that is easy to express as a series.

\begin{proposition}\label{nobdd} If $\phi$ is the symbol of a bounded multiplication operator over $PL^2$, then $\phi$ is constant.\end{proposition}

\begin{proof}
Since $M_\phi$ is a bounded multiplication operator over $PL^2$ and $z \in PL^2$, a representation of $\phi$ can be obtained immediately since $\phi z \in PL^2$. This means $\phi(z,s) = \sum_{n=1}^\infty b_n \frac{z^{n-1}}{n^s}$ for some sequence $\{ b_n \}_{n\in\mathbb{N}} \in l^2$.

For any fixed $z_0 \in \mathbb{D}$, $h_{z_0}(s) = \phi(z_0,s)$ is an entire function. Moreover, $h_{z_0}(s)$ is bounded in magnitude by $\|M_\phi\|$. Thus $h_{z_0}(s)$ is a constant function for each $z_0 \in \mathbb{D}$. Thus $\phi(z,s)$ is constant with respect to $s$. Thus $b_n = 0$ for all $n \ge 1$, and $\phi(z,s)$ is constant with respect to both variables.
\end{proof}

Proposition \ref{nobdd} will be subsumed by Theorem \ref{noddm} below. The proof of Proposition \ref{nobdd} is more asthetically pleasing than than of Theorem \ref{noddm} in that it uses the analyticity of the space to arrive at a conclusion. The proof of Theorem \ref{noddm} is more algebraic in nature, relying on the multiplication properties of Dirichlet series.

\begin{theorem}\label{noddm}The only densely defined multiplication operators over $PL^2$ are the constant functions.\end{theorem}

\begin{proof}Suppose that $\phi(z,s)$ is a densely defined multiplication operator over $PL^2$. Let $K$ be a compact subset of $\mathbb{D}\setminus\{0\}$, $f(z,s) = \sum_{n=1}^\infty a_n z^n n^{-s} \in D(M_\phi)$. Further, it may be assumed that $a_1 \neq 0$, since if there were no such function with $a_1 \neq 0$, then $D(M_\phi)$ would not be dense in $PL^2$. Now let $\sigma = Re(s) \ge \sigma_K$ be the half plane described in Lemma \ref{cn_prop}. Set $h(z,s) = \phi(z,s) \cdot f(z,s) \in PL^2$. Thus $(f(z,s))^{-1} = \sum_{n=1}^\infty \alpha_n(z) n^{-s}$ in this half plane and $\phi(z,s) = h(z,s)/f(z,s)$.

In particular, $\phi(z,s) = h(z,s)/f(z,s) = \sum_{n=1}^\infty \phi_n(z) n^{-s}$ for all $z \in K$, where $\phi_n(z)$ is some function on $K$. We will demonstrate that $\phi$ is a constant function by taking advantage of the incompatability of Dirichlet convolution and the convolution of the coefficients of power functions. The proof proceeds by induction on the number of prime factors of $n$. Here it will be shown that each coefficient $\phi_n(z)$ with $n>0$ is zero. Along the way it will be established that $\phi_{2k}(z) = c_{2k}z^{2k-1}$, but $c_{2k}$ will turn out to be zero in the next step of induction.

\begin{remark}\label{rem_ineq}For all $n$ not prime, in the proof below it is essential to notice that if $k=p_1 \cdot p_2 \cdots p_m$ is the prime factorization for $n$, then $n+1 > 2 \frac{k}{p_i} + p_i - 1$. This will isolate a term in the polynomials, and the result will follow.\end{remark}

We begin the induction at $n=1$, a number with no prime factors. Using Dirichlet convolution we find that $b_1 z^1 = \phi_1(z) a_1 z^1$. Thus $\phi_1(z) = b_1/a_1 := c_1$.

For prime $n$ we begin with $n=2$. Again via Dirichlet convolution it can be seen that $b_2z^2 = \phi_2(z) a_1 z^1 + \phi_1(z) a_2 z^2$. Since $\phi_1(z)$ is a constant function, this means that $\phi_2(z) = c_2z^1$ for some $c_2$.

If we take $n=4$ we find that $$b_4 z^4 = \phi_4(z) a_1 z^1 + \phi_2(z) a_2 z^2 + \phi_1(z) a_4 z^4 = \phi_4(z) a_1 z^1 +c_2a_2z^ + c_1a_4z^4.$$ Thus $\phi_4(z) = \frac{b_4 - c_1a_4}{a_1} z^3 - \frac{c_2a_2}{a_1} z.$ However, since $M_\phi$ is densely defined there is another function $g(z,s) = \sum_{n=1}^\infty d_n z^n n^{-s}$ for which $d_1 \neq 0$ and $d_2/d_1 \neq a_2/a_1$. Using the same algorithm we would find that the $z$ coeficient is $c_2d_2d_1^{-1}$. Since $\phi_4(z)$ is a fixed polynomial, we must have $c_2d_2d_1^{-1}=c_2a_2a_1^{-1}$. Therefore we must conclude that $c_2 = 0$ and $\phi_4(z) = c_4 z^3$. In the same manner we can show that $\phi_p(z) = 0$ for every prime $p$ and $\phi_{2p}(z)=c_{2p}z^{2p-1}$. Thus $\phi_n(z) = 0$ for each $n$ with one prime factor, and $\phi_{2n}(z) = c_{2n}z^{2n-1}$.

Suppose for each $m < m_0$ and $k=p_1\cdot p_2 \cdots p_m$ (each $p_i$ is a not necessarily distinct prime) we have $\phi_k(z)=0$, and for those $k$ with $m_0 - 1$ prime factors we have $\phi_{2k}(z) = c_{2k} z^{2k-1}$.

Take $k^\prime = p_1 \cdot p_2 \cdots p_{m_0}$. Then by our induction assumption $$b_{k'} z^{k'} \phi_{k'}(z) a_1 z^1 + c_1 a_{k'}z^{k'},$$ which yields $\phi_{k'} = c_{k'}z^{k'-1}$.

Now consider the $\phi_{2k'}$. Again using our induction assumption, the only terms that remain in the convolution are those that have $m_0$ or more prime factors and of course $\phi_1(z)$. Thus after the Dirichlet convolution:
$$\begin{array}{rcl}b_{2k'}z^{2k'}&=&\phi_{2k'}(z) a_1 z^1 + \phi_{k'}a_2z^2 + \phi_{2k'/p_1}(z)a_{p_1}z^{p_1} + \cdots\\
&+&\phi_{2k'/p_{m_0}}(z)a_{p_n}z^{p_n}+\phi_1(z)a_{2k'}z^{2k'}\\
&=&\phi_{2k'}(z)a_1 z^1 + c_{k'}a_2z^{k'+1}+c_{2k'/p_1}a_{p_1}z^{(2k'/p_1)+p_1-1}+\cdots\\
&+&c_{2k'/p_{m_0}}a_{p_{m_0}}z^{(2k'/p_{m_0})+p_{m_0}-1}+c_1a_{2k'}z^{2k'}.\end{array}$$

The function $\phi_{2k'}(z)$ can be solved for which yields: $$\phi_{2k'}(z) = \frac{b_{2k'}-c_1 a_{2k'}}{a_1} z^{2k'-1} - \frac{c_{k'}a_2}{a_1}z^{k'+1} - \text{other terms.}$$

Since none of the other terms of the function $\phi_{2k'}(z)$ has the factor $z^{k'+1}$ (by the remark above), the $k'+1$ coefficient $c_{k'}a_2a_1^{-1}$ which depends on $f$. We conclude that $c_{k'} = 0$ for all $k'$ with $m_0$ prime factors. Moreover, $\phi_{2k'}(z) = c_{2k'}z^{2k'-1}$.

Therefore by strong induction we conclude that $\phi_k(z) = 0$ for all $k\neq 1$, and $\phi(z,s)$ is constant for all $z$ in the compact set $K$. This constant does not depend on the choice of $K$, so $\phi(z,s) = b_1/a_1$ for all $s \in \mathbb{C}$ and $z \in \mathbb{D}\setminus\{0\}$. Finally by continuity $\phi(0,s) = b_1/a_1$ and that completes the theorem.\end{proof}

\begin{corollary}\label{notinvariant}No multiplication operator (with nonconstant symbol) on the tensor product, $H^2 \times \mathcal{H}^2$, has $PL^2$ as an invariant subspace.\end{corollary}

\section{Toeplitz Compressions in the Polylogarithmic Hardy Space}
\label{sec_toep}
Corollary \ref{notinvariant} mentions that $PL^2$ is not an invariant subspace for any multiplication operator over $H^2\times \mathcal{H}^2$. The space $H^2 \times \mathcal{H}^2$ does have a rich collection of bounded multiplication operators. Multiplication of a function in $PL^2$ by a multiplication operator in $H^2 \times \mathcal{H}^2$ will often result in a function in $H^2 \times \mathcal{H}^2$ but not in $PL^2$. However, since $PL^2$ is a closed subspace of $H^2 \times \mathcal{H}^2$, the projection operator is well defined. This section presents a version of Toeplitz operators, where a multiplication operator over $H^2 \times \mathcal{H}^2$ is followed by projection to $PL^2$.

\begin{definition}Given $M_\phi$ is a bounded multiplication operator over $H^2 \times \mathcal{H}^2$, the Toeplitz operator over $PL^2$ with symbol $\phi$ is given by $T_\phi = P_{PL^2} M_\phi$.\end{definition}

\begin{proposition}The following properties hold for Toeplitz operators over $PL^2$:
\begin{enumerate}
\item $T_{z^k/m^s}$ is the zero operator iff $m-1 \not\ \mid k$. If $m-1 \mid k$ then this is a rank one operator.
\item $T_{z^k\zeta(s)}$ has rank $d(k) = \sum_{k_0 \mid k} 1$.
\item The Toeplitz operator with symbol $\phi(z,s) = \sum_{n=0, m=1}^\infty c_n z^n m^{-s}$ is compact when $\sum_{n,m} |c_{n,m}| < \infty$.
\end{enumerate}
\end{proposition}

\begin{proof}
For the first part, let $f(z,s)=\sum_{n=1}^\infty a_n z^n n^{-s} \in PL^2$, $k \in \mathbb{N}\cup \{0\}$, $m \in \mathbb{N}$. Consider the product $\frac{z^k}{m^s} f(z,s)= \sum_{n=1}^\infty a_n \frac{z^{n+k}}{(mn)^s}$. The only terms that will survive the projection will be those for which $n+k=mn$ and so $n(m-1) = k$. Therefore if $m-1 \not\ \mid k$, $T_{z^k/m^s}$ is the zero operator. If $m-1 \mid k$, then $T_{z^k/m^s} \frac{z^{n}}{n^s} = \frac{z^{n+k}}{(n+k)^s}$ if $n = k/(m-1)$ and the result is zero otherwise.

Determining the rank of $T_{z^k\zeta(s)}$ is straightforward given the observations of the last paragraph. Considering the basis functions for $PL^2$, $\{ z^n/n^s\}$ only those for which $n = k/(m-1)$ for some $m \in \mathbb{N}$ will not be sent to zero. Each surviving basis function will be mapped to a distinct basis function, $z^{(n+k)}{(n+k)^s}$. Therefore there is a contribution to the rank for each divisor $m-1 \mid k$, of which there are $d(k)$.

Finally, given $\phi(z,s) = \sum_{n=0,m=1}^\infty c_{n,m} z^n m^{-s}$, the compactness of $T_\phi$ can be expressed by writing $T_\phi$ as a limit of finite rank operators. From the discussion above it can be seen that given $\phi_{n_0} = \sum_{m=1}^\infty c_{n_0,m}z^{n_0}m^{-s}$ has rank at most $d(n_0)$. We can decompose the function $\phi$ as $\phi = \sum_{n=0}^\infty \phi_n$. The goal is to show that $T_\phi = \sum_{n=0}^\infty T_{\phi_n}$.

First note that $$T_{\phi_{n_0}} f(z,s) = \sum_{m-1 \mid n} c_{n_0,m} a_{n_0/(m-1)} \frac{z^{n_0 m/(m-1)}}{(n_0m/(m-1))^s}.$$ An application of Cauchy-Schwarz leads to the conclusion that $$\| T_{\phi_{n_0}} \| \le \sqrt{ \sum_{m-1 \mid n_0} |c_{n_0,m}|^2} \le \sum_{m-1 \mid n_0} |c_{n_0,m}| := C_{n_0}.$$

Now we can see that $\sum_{n=0}^\infty C_{n} \le \sum_{n=0,m=1}^\infty |c_{n,m}| < \infty$. Thus the series $\sum_{n=0}^\infty T_{\phi_{n}}$ is absolutely convergent and hence convergent. The final verification amounts to checking that both $\sum_{n=0}^\infty T_{\phi_n}$ and $T_\phi$ act identically on the orthonormal basis $z^n/n^s$.
\end{proof}

The appearance of the arithmetic function $d(k)$ in the rank of the operator $T_{z^k \zeta(s)}$ is suggestive of the structure of the operator. The function $d(k)$ has the multiplicative property: $d(k_1 k_2) = d(k_1) d(k_2)$ whenever $gcd(k_1, k_2) = 1$. Therefore from number theory we can see that $\rank(T_{z^{k_1 k_2}\zeta(s)}) = \rank(T_{z^{k_1}\zeta(s)}) \rank(T_{z^{k_2}\zeta(s)})$ when $k_1$ and $k_2$ are coprime. Theorem \ref{decomp} below demonstrates this through an explicit decomposition of $T_{z^{k_1k_2}\zeta(s)}$ and thereby proves independently the multiplicativity of $d(k)$.

First we give a couple of examples. Consider the matrix representation of $T_{z^2\zeta(s)}$, $T_{z^3\zeta(s)}$, $T_{z^4\zeta(s)}$, and $T_{z^6\zeta(s)}$:
$$\left(\begin{array}{cc}
0&0\\
0&0\\
1&0\\
0&1\end{array}\right),
\left(\begin{array}{ccc}
0&0&0\\
0&0&0\\
0&0&0\\
1&0&0\\
0&0&0\\
0&0&1\end{array}\right),
\left(\begin{array}{cccc}
0&0&0&0\\
0&0&0&0\\
0&0&0&0\\
0&0&0&0\\
1&0&0&0\\
0&1&0&0\\
0&0&0&0\\
0&0&0&1\end{array}\right), \text{ and}
\left( \begin{array}{cccccc}
0&0&0&0&0&0\\
0&0&0&0&0&0\\
0&0&0&0&0&0\\
0&0&0&0&0&0\\
0&0&0&0&0&0\\
0&0&0&0&0&0\\

1&0&0&0&0&0\\
0&1&0&0&0&0\\
0&0&1&0&0&0\\
0&0&0&0&0&0\\
0&0&0&0&0&0\\
0&0&0&0&0&1\end{array}\right)$$
respectively. The lower half of each matrix can be thought of as the divisibility matrix of $k$ in $T_{z^k\zeta(s)}$. That is, diagonal entry, $i$, is 1 if $i \mid k$ and zero otherwise. We can see that the matrix representation of $T_{z^2\zeta(s)}$ appears in that of $T_{z^6\zeta(s)}$. In fact, it appears twice. Once as itself identically. The second appearance is a stretched version of the matrix (scaled by 3). This motivates Theorem \ref{decomp}.

There are two natural shift operators in $PL^2$. The first comes from the Hardy space, for each $m \in \mathbb{N}$, $z^{n}/n^s \mapsto z^{n+m}/(n+m)^s$ which will be denoted $S_{+m}$. The second comes from the Hardy space of Dirichlet series, for each $m$ there is a shift operator of the form $z^{n}/n^s \mapsto z^{mn}/(mn)^s$ which will be denoted $S_{\times m}$. In the case of the Hardy space the shift operator arises from multiplication by the independent variable. In the Hardy space of Dirichlet series, these shift operators arise from multiplication by $m^{-s}$ \cite{Olofsson}. Of course, by way of Theorem \ref{noddm}, the shift operators do not have corresponding multiplication operators in $PL^2$. 

\begin{theorem}\label{decomp}When $n$ and $m$ are coprime, the operator $T_{z^{nm} \zeta(s)}$ can be decomposed as follows: $$T_{z^{nm}\zeta(s)} = \bigoplus_{k | m} S_{+(m+n)} S^{*}_{+kn} S_{\times k} T_{z^n \zeta(s)} S^{*}_{\times k} P_k$$ where $P_k$ is the projection onto the span of the vectors $z^{k\cdot r}/(k\cdot r)^s$ for $r \mid n$.\end{theorem}

\begin{proof}It is sufficient to check that $P_k$ and $P_{k'}$ are projections on orthogonal spaces. The rest of the theorem amounts to checking the operators' action on the basis functions.

If there was a basis function $e_{n_0} = z^{n_0}/n_0^s$ for which $P_ke_{n_0} = P_{k'}e_{n_0} = e_{n_0}$ then by definition $n_0 = k r = k' r'$ for some $r, r' \mid n$. However, since $gcd(k,n) = gcd(k',n) = 1$, we have $k' \mid k$ and $k \mid k'$. Therefore $k = k'$. Thus $P_k$ and $P_{k'}$ are projections on the same subspace iff $k=k'$. 
\end{proof}

\begin{corollary}When $n$ and $m$ are coprime, $d(nm)=d(n)d(m)$.\end{corollary}
\begin{proof}$$\begin{array}{rcl} \rank(T_{z^{nm}\zeta(s)})
&=& \sum_{k \mid m} \rank( S_{+(m+n)} S^{*}_{+kn} S_{\times k} T_{z^n \zeta(s)} S^{*}_{\times k} P_k)\\
&=&  \sum_{k\mid m} \rank(T_{z^n \zeta(s)}) = \sum_{k\mid m} d(n) = d(m) d(n).\end{array}$$\end{proof}

\bibliography{polylogarithm}{}

\begin{thebibliography}{10}

\bibitem{AgMc}
Jim Agler and John~E. McCarthy.
\newblock {\em Pick interpolation and {H}ilbert function spaces}, volume~44 of
  {\em Graduate Studies in Mathematics}.
\newblock American Mathematical Society, Providence, RI, 2002.

\bibitem{apostol}
Tom~M. Apostol.
\newblock {\em Introduction to analytic number theory}.
\newblock Springer-Verlag, New York-Heidelberg, 1976.
\newblock Undergraduate Texts in Mathematics.

\bibitem{Bayart}
Frederic Bayart.
\newblock Hardy spaces of dirichlet series and their composition operators.
\newblock {\em Monatshefte fur Mathematik}, 136(3):203--236, 2002.

\bibitem{zagier3}
Henri Cohen, Leonard Lewin, and Don Zagier.
\newblock A sixteenth-order polylogarithm ladder.
\newblock {\em Experiment. Math.}, 1(1):25--34, 1992.

\bibitem{garnett}
John~B. Garnett.
\newblock {\em Bounded analytic functions}, volume 236 of {\em Graduate Texts
  in Mathematics}.
\newblock Springer, New York, first edition, 2007.

\bibitem{hardy}
G.~H. Hardy and E.~M. Wright.
\newblock {\em An introduction to the theory of numbers}.
\newblock Oxford University Press, Oxford, sixth edition, 2008.
\newblock Revised by D. R. Heath-Brown and J. H. Silverman, With a foreword by
  Andrew Wiles.

\bibitem{hedenmalm}
H{\aa}kan Hedenmalm, Peter Lindqvist, and Kristian Seip.
\newblock A {H}ilbert space of {D}irichlet series and systems of dilated
  functions in {$L^2(0,1)$}.
\newblock {\em Duke Math. J.}, 86(1):1--37, 1997.

\bibitem{hewitt}
Edwin Hewitt and J.~H. Williamson.
\newblock Note on absolutely convergent {D}irichlet series.
\newblock {\em Proc. Amer. Math. Soc.}, 8:863--868, 1957.

\bibitem{hoffman}
Kenneth Hoffman.
\newblock {\em Banach spaces of analytic functions}.
\newblock Dover Publications, Inc., New York, 1988.
\newblock Reprint of the 1962 original.

\bibitem{lewin}
Leonard Lewin.
\newblock {\em Polylogarithms and associated functions}.
\newblock North-Holland Publishing Co., New York-Amsterdam, 1981.
\newblock With a foreword by A. J. Van der Poorten.

\bibitem{mccarthy}
John~E. McCarthy.
\newblock Hilbert spaces of dirichlet series and their multipliers.
\newblock {\em Trans. Amer. Math. Soc.}, 356(3):881--893 (electronic), 2004.

\bibitem{Olofsson}
A.~Olofsson.
\newblock On the shift semigroup on the hardy space of dirichlet series.
\newblock {\em Acta Mathematica Hungarica}, 128(3):265--286, 2010.

\bibitem{statmech}
Paul D.~Beale R.~K.~Pathria.
\newblock {\em Statistical Mechanics}.
\newblock Academic Press, 2011.

\bibitem{riesz}
Freidrich Riesz.
\newblock \"{U}ber die {R}andwerte einer analytischen {F}unktion.
\newblock {\em Math. Z.}, 18(1):87--95, 1923.

\bibitem{rosenfeld}
Joel~A. Rosenfeld.
\newblock Densely defined multiplication on several sobolev spaces of a single
  variable.
\newblock {\em Complex Analysis and Operator Theory}, pages 1--7, 2014.

\bibitem{sarason}
Donald Sarason.
\newblock Unbounded {T}oeplitz operators.
\newblock {\em Integral Equations Operator Theory}, 61(2):281--298, 2008.

\bibitem{frank}
FranciszekHugon Szafraniec.
\newblock The reproducing kernel hilbert space and its multiplication
  operators.
\newblock In EnriqueRamírez de~Arellano, NikolaiL. Vasilevski, Michael
  Shapiro, and LuisManuel Tovar, editors, {\em Complex Analysis and Related
  Topics}, volume 114 of {\em Operator Theory Advances and Applications}, pages
  253--263. Birkhäuser Basel, 2000.

\bibitem{zagier2}
Don Zagier.
\newblock The {B}loch-{W}igner-{R}amakrishnan polylogarithm function.
\newblock {\em Math. Ann.}, 286(1-3):613--624, 1990.

\bibitem{zagier1}
Don Zagier.
\newblock The dilogarithm function.
\newblock In {\em Frontiers in number theory, physics, and geometry. {II}},
  pages 3--65. Springer, Berlin, 2007.

\bibitem{zhu}
Kehe Zhu.
\newblock {\em Analysis on {F}ock spaces}, volume 263 of {\em Graduate Texts in
  Mathematics}.
\newblock Springer, New York, 2012.

\end{thebibliography}
\bibliographystyle{plain}
\end{document}